\documentclass{amsart}
\usepackage[utf8]{inputenc}
\usepackage[T1]{fontenc}
\usepackage[english]{babel}
\usepackage{color}
\usepackage{amsmath} 
\usepackage{amssymb} 
\usepackage{amsthm}
\usepackage{graphicx}
\usepackage[colorlinks=true, linkcolor=blue]{hyperref}
\usepackage{cancel}

\usepackage{lmodern}
\usepackage{microtype}

\theoremstyle{plain}
\newtheorem{thm}{Theorem}
\newtheorem*{thm*}{Theorem}
 
\newtheorem{lem}[thm]{Lemma} 

\newtheorem*{cor*}{Corollary}

\newtheorem{defi}[thm]{Definition}

\newtheorem{rem}[thm]{Remark}

\newcommand {\R} {\mathbb{R}} \newcommand {\Z} {\mathbb{Z}}
\newcommand {\T} {\mathbb{T}} \newcommand {\N} {\mathbb{N}}

\newcommand {\p} {\partial}
\newcommand {\dt} {\partial_t}

\newcommand {\dist} {\text{dist}}

\begin{document}
\title[Resonance Chains]{On the Smallness Condition in Linear Inviscid Damping: Monotonicity and
  Resonance Chains}
\author{Yu Deng}
\author{Christian Zillinger}
\address{USC Dornsife, Department of Mathematics, 3629 S. Vermont Ave., CA 90089
  Los Angeles, USA}
\email{yudeng@usc.edu}
\address{BCAM - Basque Center for Applied Mathematics, Mazarredo, 14 E48009 Bilbao, Basque Country – Spain}
\email{czillinger@bcamath.org}

\begin{abstract}
  We consider the linearized Euler equations around a smooth, bilipschitz shear
  profile $U(y)$ on $\T_L \times \R$.
  We construct an explicit flow which exhibits linear inviscid damping for $L$ sufficiently small, but for which
  damping fails if $L$ is large. In particular, similar to the instability
  results for convex profiles for a shear flow being bilipschitz is not
  sufficient for linear inviscid damping to hold.
  Instead of an eigenvalue-based argument the underlying mechanism here is shown
  to be based on a new cascade of resonances
  moving to higher and higher frequencies in $y$, which is distinct from the echo chain
  mechanism in the nonlinear problem. 
\end{abstract}

\maketitle

\section{Introduction}
\label{sec:intro}

We are interested in the long-time asymptotic behavior of the linearized 2D
Euler equations near monotone shear flows $v=(U(y),0)$ in a periodic channel
$\T_L\times \R$ or circular flows $v= rU(r)e_{\theta}$ on $\R^2$. After possibly
relabeling the log-polar coordinates and considering weighted spaces (see
\cite{coti2019degenerate}, \cite{Zill6}) both settings can be considered in the framework
\begin{align*}
  \dt \omega + U(y)\p_x \omega - \beta(y)\p_x \Delta^{-1} =0,
\end{align*}
where $\beta(y)=U''(y)$ in the plane channel setting.

One observes that if $\beta(y)\equiv 0$ (corresponding to (Taylor-)Couette
flow) the equation has an explicit solution $\omega(t,x,y)=\omega(0,x-tU(y),y)$,
which weakly converges to its $x$-average. As a consequence the velocity field
strongly converges to an asymptotic profile as $t\rightarrow \infty$. This
phenomenon is known as (linear) inviscid damping. While results for this special
case follow by explicit calculation, the study of the asymptotic behavior of
non-trivial flows has been an area of active research in recent years, where a
guiding question has been to understand how robust this mechanism might be.

In this article, we construct a negative example in the form of a small sine wave perturbation to a linear shear
\begin{align*}
  U(y)=y+ c \sin(y) 
\end{align*}
with $|c|<\frac{1}{2}$. More precisely, for simplicity of calculation we
consider the approximate system $U(y)=y$ and $\beta(y)=2c\sin(y)$, see Section \ref{sec:Duhamel}.
We stress that $U(y)$ is Bilipschitz and smooth and that $\beta(y)$ is small and analytic.
Our choice of $U(y)=y$ is motivated by the existing linear
and nonlinear damping results for Couette flow, where strict monotonicity and the
associated shearing mechanism serves to stabilize the flow, but require either
smallness conditions (see for example \cite{coti2019degenerate}, \cite{Zill5}) or non-resonance conditions
(see for example \cite{WZZkolmogorov}).

A further, separate stabilizing mechanism is given by convexity and sign
conditions such as those used in the classical stability results of Rayleigh,
Fjortoft and Arnold. 
In this view, we further mention for instance the article of Bedrossian, Coti Zelati and Vicol
\cite{BCZVvortex2017} on circular flows, where monotone, convex profiles are
considered and where no smallness condition is required.
However, we remark that convexity by itself is only sufficient given a
beneficial sign. Indeed, in the sense of Arnold's theory Kolmogorov flow,
$U(y)=\cos(y)$, is strictly convex with $\frac{U''}{U}=-1$, but is known to be
nonlinearly stable only on a short torus $\T_L \times \T$,
$L<1$ but unstable on a long torus $L\gg 1$, \cite{meshalkin1961investigation}.

Our main question of this article concerns the robustness of the stabilization
by shearing for Bilipschitz, non-convex profiles for short and long tori, in analogy to the
results on Kolmogorov flow as a convex flow.
Our main results are summarized in the following theorem.

\begin{thm}
  \label{thm:summary}
  Let $c \in (0,\frac{1}{2})$ and consider the following linear problem
  \begin{align*}
    \dt \omega + y\p_x \omega -c\sin(y)\p_x\Delta^{-1}\omega=0,
  \end{align*}
  on the domain $\T_{L} \times \R$.
  Then if $cL <\frac{1}{2\pi}$ it holds that for any $s\geq 0$ the evolution is stable in
  $H^s$ in the sense that there exists $C_s$ such that for any time $t\geq 0$
  \begin{align*}
    \|\omega(t)\|_{H^s}\leq C_s \|\omega|_{t=0}\|_{H^s}.
  \end{align*}
  Furthermore, there exists $C_0$ such that if
  \begin{align*}
    \|\omega|_{t=0}\|_{\mathcal{G}_{1,C}}^2=\int \exp(C|\eta|) |\tilde{\omega}|_{t=0}(\eta)|^2 <\infty 
  \end{align*}
  for some $C>C_0$, then also $\|\omega(t)\|_{\mathcal{G}_{1,C-C_0}}<\infty$ for all
  times.

  On the other hand, suppose that $c<\frac{1}{10}$ and $L$ is such that $c \pi
  L>20$, but $ \pi c^2 L <1$ (e.g. choose $L \sim c^{-3/2}$). Then there exists smooth initial data $\omega_0$
  (compactly supported in Fourier space),
  a frequency $k$ and $d>1$ such that its Fourier transform satisfies
  \begin{align*}
    |\tilde{\omega}(t_j,k,\eta=kt_j)|\geq d^{t_j}
  \end{align*}
  for all $t_j=j \in \N$. That is, the evolution is
  exponentially growing along a sequence of times.
\end{thm}

The growth mechanism in the second result here is not given by an eigenfunction
construction, but given by a new \emph{cascade mechanism} which propagates in the
frequency $\eta$ associated to $y$.
We stress that this resonance chain is distinct from the one considered in
\cite{deng2018}, \cite{bedrossian2013inviscid} which propagates in the frequency
$k$ associated to $x$ (see Section \ref{sec:instability}).

For simplicity of calculation and presentation we neglect the effects of
boundaries and consider the setting of an infinite periodic channel
$\T_{L}\times \R$.
We however expect that our results should extend with moderate technical effort
to the setting of compactly supported perturbations to Couette flow (which is
considered in \cite{jia2019linear}) in the case of Gevrey regularity and to more general
Bilipschitz flows as considered in \cite{Zill5}.

Our article is organized as follows:
\begin{itemize}
\item In Section \ref{sec:Duhamel} we first establish global well-posedness of the
  equations and some rough upper growth bounds.
  Afterwards we establish the stability statement of Theorem \ref{thm:summary} by
  means of a Duhamel iteration estimate. We further present a second proof by
  means of a Lyapunov functional, which is very transparent but requires
  stronger assumptions.
\item In Section \ref{sec:instability} we analyze the resonance mechanism in
  detail and how it can excite neighbors of resonant modes. As our main result
  we show that if $cL$ is sufficiently large this growth can be sustained,
  resulting in an exponential growth rate. In particular, this shows that it is
  not sufficient for a flow to be smooth and bilipschitz in order for asymptotic
  stability and linear inviscid damping to hold, but that some further smallness
  or non-resonance condition is required.
  While some eigenvalue instability constructions are available in the
  literature \cite{Zhang2015inviscid}, we here present a new cascade mechanism underlying this sustained
  growth as well an explicit example.
\end{itemize}

\subsubsection*{Acknowledgments}
Yu Deng acknowledges support by NSF grant DMS-1900251.

Christian Zillinger would like to thank the Max-Planck Institute for Mathematics
in the Sciences where part of this work was written for its hospitality.

Christian Zillinger's research is supported by the ERCEA under the grant 014
669689-HADE and also by the Basque Government through the BERC 2014-2017
program and by Spanish Ministry of Economy and Competitiveness MINECO: BCAM Severo Ochoa excellence accreditation SEV-2013-0323.

\section{Global Well-posedness and Duhamel Iterations}
\label{sec:Duhamel}

The linearized 2D Euler equations around $v=(y+c \sin(y),0)$ are given by:
\begin{align*}
  \dt \omega + (y+c\sin(y))\p_x \omega + c \sin(y) v_2=0.
\end{align*}
As the map $\omega \mapsto v_2$ is compact, we consider the transport by the
linear profile to be dominant, change to
Lagrangian coordinates $(x+ty,y)$ and consider
\begin{align*}
W(t,x,y)=:\omega(t,x+ty,y).
\end{align*}
With respect to these coordinates the equation is given by 
\begin{align*}
  \dt W + c\sin(y)\p_x W + c \sin(y) \p_x \Delta_{t}^{-1} W &=0 ,\\
  \Delta_t&:=\p_x^2+(\p_y-t\p_x)^2.
\end{align*}
Here, our choice of coordinates has the benefit of the very transparent
structure of $\Delta_t^{-1}$ as a Fourier multiplier.
As the coefficient functions do not depend on $x$, we note that this system
decouples in the associated frequency $k$, which we hence treat as arbitrary but
fixed.
Since our main focus is on considering the effect of long tori (and thus very
small frequency $k$), we further consider an approximate system
\begin{align}
   \dt W + c \sin(y) \p_x \Delta_{t}^{-1} W &=0,
\end{align}
where we omit the contribution by $c\sin(y)\p_x$ and with slight abuse of notation denote the
Fourier transform of $W$ by $\omega$ again.

This then leads to the following system:
\begin{align}
  \dt \omega(\eta) + \frac{c}{2} \frac{k}{k^2+(\eta+1-kt)^2} \omega(\eta+1) - \frac{c}{2} \frac{k}{k^2+(\eta-1-kt)^2} \omega(\eta-1)=0.
\end{align}
Since the system decouples in $k$ and leaves the mode $k=0$ invariant,
in the following we without loss of generality consider $k\neq 0$ and
equivalently rescale
time as $\tau=kt$, which yields:
\begin{align}
  \label{eq:1}
  \p_\tau \omega(\eta) +  \frac{c}{2} \frac{1}{k^2+(\eta+1-\tau)^2} \omega(\eta+1) - \frac{c}{2} \frac{1}{k^2+(\eta-1-\tau)^2} \omega(\eta-1)=0.
\end{align}
\begin{rem}
  \label{rem:periodic}
This equation is an ODE system with nearest-neighbor interaction.
In particular, since modes interact only via chains of neighbors, the system
decouples into problems on $\eta \in \eta^*+\Z$, with $\eta^* \in [0,1)$.
Considering $\eta^{*}$ as arbitrary but fixed and with slight abuse of notation
shifting time by $\eta^*$ to prescribe ``initial'' data at that time it thus
suffices to consider the \emph{periodic} problem $\eta \in \Z$.
We remark that this setting would also appear when considering periodic perturbations with
respect to Lagrangian coordinates $(x+ty,y)$, which seems physically
unmotivated, but by the above consideration appears naturally due to the
decoupling structure.
\end{rem}

Roughly estimating the coefficient functions by $\frac{c}{2k^2}$, we immediately
obtain the following suboptimal global well-posedness result.

\begin{lem}
  \label{lem:Gronwall}
  Let $X= L^2(\rho, \frac{2\pi}{L}\Z \times \R)$ be a weighted $L^2$ space in Fourier space such that its
  weight satisfies
  \begin{align*}
  \sup_{k,\eta}\frac{\rho(k,\eta\pm 1)}{\rho(k,\eta)} \leq C_1 <\infty.
  \end{align*}
  This for example includes fractional Sobolev spaces $H^s$ or Gevrey spaces.
  
  Then there exists a constant $C$ such that the solution $\omega$ of \eqref{eq:1} satisfies
  \begin{align*}
    \|\omega(\tau)\|_{X} \leq \exp(C cL^2 \tau)\|\omega_0\|_{X}
  \end{align*}
  for all $\tau\geq 0$.
\end{lem}

We remark that the time variable $\tau$ includes a factor $L$ and that expressed
with respect to $t$ the growth factor is given by $\exp (C cL t)$.
\begin{proof}
  We note that
  \begin{align*}
    \|\p_\tau \omega\|_{X}\leq \| c \frac{1}{k^2+(\eta\pm 1-\tau)^2} \omega(\eta\pm 1)\|_{X} \leq cC L^2\|\omega\|_{X}.
  \end{align*}
  The result then immediately follows by integrating the differential inequality.
\end{proof}

While this a priori bound is sufficient to establish global well-posedness, it
is far from sharp. Indeed, the following theorem shows that for $cL$ sufficiently
small, the evolution is globally stable in a Lyapunov sense. In contrast, in
Theorem \ref{thm:instability} we show that if $L \gg c^{-1}$, then this stability fails and the
evolution is exponentially unstable. Here, we stress that the shear profile is
Bilipschitz and $c \sin(y)$ is smooth and bounded. Furthermore, we show that the
described growth is due to a new cascade mechanism, where chains
of resonances excite higher and higher modes in $y$ (while the nonlinear echo
chain mechanism excites smaller and smaller modes in $x$ and stops at mode $1$).

As first result, we consider the setting of very small $c$, which is a amenable
to the construction of a Lyapunov functional as in \cite{coti2019degenerate}.
Here, the simple Fourier structure allows for a particularly transparent proof.
In a second Theorem \ref{thm:secondstability} we introduce a different method of
proof by expressing the Duhamel iteration as formal infinite series of
\emph{integrals over paths} (see Definition \ref{defi:path}). 

\begin{thm}[Global stability for small $cL$]
  \label{thm:multiplier}
  Let $j \in \N$ and consider the stability problem for the pde \eqref{eq:1} on
  $\T_{L}\times \R$ in $H^{j}$. Then there exists $C_0=C_0(j)>0$ such that if $
  cL <C_0$, there exists a constant $C$ such that
  \begin{align*}
    \|\omega(\tau)\|_{H^{j}}\leq C \|\omega_0\|_{H^{j}}
  \end{align*}
  for all $\tau\geq 0$.
  Furthermore, $\p_x \Delta_{t}^{-1}W \in L^2_tH^{j}$ and $\omega$ converges
  strongly in $H^{j}$ as $\tau \rightarrow \infty$.
\end{thm}

\begin{proof}
  We define the Fourier weight
  \begin{align*}
    a(\tau,\eta)= \exp(C_1 c\arctan(C_2(\eta-\tau)),
  \end{align*}
  with $C_1, C_2>0$ to be fixed later.
  
  Then it holds that
  \begin{align*}
    & \quad \frac{d}{d \tau} a(\tau,\eta)|\omega(\tau,\eta)|^2 \\
    &= - C_1 C_2\frac{c}{1+C_2^2(\eta-\tau)^2} a(\tau,\eta)|\omega(\tau,\eta)|^2 \\
    & \quad + 2 a(\tau,\eta) \frac{c}{k^2+(\eta+1 -\tau)^2} \omega(\tau,\eta+1) \omega(t, \eta)\\
    & \quad - 2 a(\tau,\eta) \frac{c}{k^2+(\eta-1 -\tau)^2} \omega(\tau,\eta-1) \omega(\tau, \eta).
  \end{align*}
  We may estimate $|\omega(\tau,\eta \pm 1)\omega(\tau,\eta)|\leq
  \frac{1}{2}(|\omega(\tau,\eta\pm 1)|^2 + |\omega(\tau,\eta)|^2)$ and sum over all
  $\eta$ to obtain:
  \begin{align*}
    & \quad \frac{d}{d\tau} \sum_{\eta} a(\tau,\eta)|\omega(t,\eta)|^2 \\
    &\leq - C_1 C_2 \sum_{\eta}  \frac{c}{1+C_2^2(\eta-t)^2} a(\tau,\eta)|\omega(\tau,\eta)|^2  \\
    & \quad + \sum_{\eta} a(\tau, \eta) (\frac{c}{k^2+(\eta+1 -t)^2} + \frac{c}{k^2+(\eta-1 -t)^2}) |\omega(\tau,\eta)|^2  \\
    & \quad +  \sum_{\eta} (a(\tau, \eta-1)+a(\tau, \eta+1)) \frac{c}{k^2+(\eta-\tau)^2}|\omega(\tau,\eta)|^2,
  \end{align*}
  where we shifted $\eta$ in the last sum.
By the definition of our weight $a$ it holds that
  \begin{align*}
    \left(a(\tau, \eta-1)+a(\tau, \eta+1)\right) \leq \exp(2C_1c) a(\tau,\eta)
  \end{align*}
  and that
  \begin{align*}
    (\frac{c}{k^2+(\eta+1 -\tau)^2} + \frac{c}{k^2+(\eta-1 -\tau)^2}) \leq 4 k^{-2} \frac{c}{1+k^{-2}(\eta-\tau)^2}.
  \end{align*}
  Hence, we may conclude that
  \begin{align}
    \label{eq:2}
    \frac{d}{d\tau} \sum_{\eta} a(\tau,\eta)|\omega(\tau,\eta)|^2 \leq - \frac{C}{2} \sum_{\eta}  \frac{c}{k^2+(\eta-\tau)^2} a(\tau,\eta)|\omega(\tau,\eta)|^2 \leq 0,
  \end{align}
  provided that
  \begin{align}
    \label{eq:3}
    C_2=k^{-1}, \\
    C_1C_2 - 2 k^{-2}\exp(2C_1c)\geq \frac{C}{2}.
  \end{align}
  We may then fix $C_2=k^{-1}\leq L$, $C_1=4 k^{-1}$, at which point the estimate reduces to
  $4-2\exp(2cL)=: C/2 >0$, which is satisfied if $cL$ is sufficiently small.

  We note that \eqref{eq:2} implies $l^2$ stability, since $a(\tau, \eta)\approx 1$.
  Furthermore, as $\sum_{\eta} a(\tau,\eta)|\omega(t,\eta)|^2 |_{t=0}^T$ is
  bounded, it follows that
  \begin{align*}
    \sum_{\eta}  \frac{c}{k^{2}+(\eta-t)^2} a(\tau,\eta)|\omega(\tau,\eta)|^2 
  \end{align*}
  is integrable in time. Expressing the evolution equation in integral form then
  further yields the claimed asymptotic stability in $L^2$.

  It remains to establish stability in higher Sobolev norms.
  We define
   \begin{align*}
    a_j(\tau, \eta):= \langle \eta \rangle^j a(\tau,\eta)
   \end{align*}
   for $j \in \N$ and further define
   \begin{align*}
     A_j(\tau,\eta): =  a_{j}(\tau,\eta)+ \sum_{j'\leq j-1} {j \choose j'} A_{j'}(\tau,\eta).
   \end{align*}
   Using \eqref{eq:2} as the start of a proof by induction, we then claim that
   \begin{align}
     \label{eq:4}
     \frac{d}{d\tau} \sum_{\eta} A_{j}(\tau,\eta) |\omega(\eta)|^2 \leq - \frac{C}{2} \sum_{\eta}\frac{c}{k^2+(\eta-\tau)^2}A_{j}(\tau,\eta) |\omega(\eta)|^2.
   \end{align}
   Let thus $\hat{j}\geq 1$ be given and suppose that \eqref{eq:4} is satisfied
   for all $j\leq \hat{j}$.
   Then it holds that
   \begin{align*}
     &\quad \frac{d}{dt}\sum_{\eta} A_{\hat{j}+1}(\tau,\eta) |\omega(\eta)|^2 \\
     &=  \frac{d}{dt}\sum_{\eta} a_{\hat{j}+1}(\tau,\eta) |\omega(\eta)|^2 + 2\frac{d}{dt} \sum_{\eta} A_{\hat{j}}(\tau,\eta) |\omega(\eta)|^2 \\
     &\leq -C \sum_{\eta} \frac{c}{1+(\eta-\tau)^2} a_{\hat{j}+1}(\tau,\eta) |\omega(\eta)|^2 \\
     & \quad +  2\sum_{\eta} a_{\hat{j}+1}(\tau,\eta) \omega(\tau,\eta) \frac{d}{d\tau}\omega(\tau,\eta) \\
     & - C \sum_{\eta}\frac{c}{1+(\eta-\tau)^2}A_{\hat{j}}(\tau,\eta) |\omega(\eta)|^2.
   \end{align*}
 Plugging in the equation for $ \frac{d}{dt}\omega(\tau,\eta)$ and using
 Young's inequality, we again have to control shifts:
  \begin{align*}
    a_{\hat{j}+1}(\tau,\eta\pm 1) &= \langle \eta \pm 1 \rangle^{\hat{j}+1} a(\tau,\eta\pm 1) \\
                     &\leq \frac{a(\tau,\eta \pm 1)}{a(\tau,\eta)} a_{\hat{j}+1}(\tau,\eta) +  (\langle \eta \pm 1 \rangle^{\hat{j}+1} - \langle \eta \rangle^{\hat{j}+1}) a(\tau,\eta \pm 1) \\
                     &\leq \exp(cC) a_{\hat{j}+1}(\tau,\eta) + \sum_{j'\leq \hat{j}}{j \choose j'}  a_{j'}(\tau,\eta \pm 1),
  \end{align*}
  where the binomial factors are obtained by expanding $\langle \eta \pm 1
  \rangle^{\hat{j}+1} - \langle \eta \rangle^{\hat{j}+1}$.
  As in the case $\hat{j}=0$, $\exp(cC) a_{\hat{j}+1}(\tau,\eta)$ may be absorbed by $\frac{d}{d\tau} a_{\hat{j}+1}(\tau,\eta)$
  provided $c$ is sufficiently small. The additional correction by ${j \choose j'}  a_{j'}(\tau,\eta \pm 1)$ is of lower order and can be absorbed into
  \begin{align*}
    - C \sum_{\eta}\frac{c}{k^2+(\eta-\tau)^2}A_{\hat{j}}(\tau,\eta) |\omega(\eta)|^2.
  \end{align*}
\end{proof}

While the preceding method of proof allows for a useful control of the
evolution and stability and is very explicit, it is only applicable in the regime of very small $cL$
and does not allow for a more precise description of the evolution in terms of
the initial data.

In the following we hence instead argue by considering iterated
Duhamel iterations as a formal infinite series. We remark that in \cite{dengZ2019} we employed
similar methods to study the (nonlinear) echo chain mechanism.

Lemma \ref{lem:Duhamel} establishes the convergence of the Duhamel iteration under very mild assumptions.
In the following we then relate the Duhamel iteration to considering sums over
\emph{paths} (see Definition \ref{defi:path}). Theorem \ref{thm:secondstability} serves
to introduce our new techniques in a transparent and accessible manner, also for
$c$ small. Afterwards, in Theorem \ref{thm:instability} we show that for $c$
larger than this, but still smaller than $0.5$, a cascade of resonances yields exponential growth and a well-described non-trivial asymptotic behavior.

\begin{lem}
  \label{lem:Duhamel}
  Let $U(y)$ be a measurable function and let $\beta \in L^\infty$.
  Consider the problem
  \begin{align*}
    \dt \omega + \beta(y) \p_x \omega \tilde{\Delta}_t^{-1} \omega =0,
  \end{align*}
  where $\tilde{\Delta}_t$ is obtained by conjugating $\Delta$ with the flow with
  $(U(y),0)$.
  Since the equations preserve the $x$ average, additionally assume without
  loss of generality that $\int \omega_0dx =0$.
  Then there exists a constant $C$ possibly depending on the domain such that for all
  times $t\geq 0$ it holds that 
  \begin{align*}
    \|\omega(t)\|_{L^2}^2 \leq e^{C\|\beta\|_{L^\infty} t} \|\omega_0\|_{L^2}^2.
  \end{align*}
  Furthermore, the formal infinite Duhamel series
  \begin{align*}
    \text{Id} + \int_0^t \beta(y) \p_x \omega \tilde{\Delta}_{t_1}^{-1} d t_{1} \\
    + \int_0^t \int_0^{t_1}\beta(y) \p_x \omega \tilde{\Delta}_{t_1}^{-1}
    \beta(y) \p_x \omega \tilde{\Delta}_{t_2}^{-1} d t_{2} d t_1 + \dots
  \end{align*}
  is convergent in the $L^2$ operator topology with the $j^{th}$-term bounded by $\frac{(C\|\beta\|_{L^\infty}t)^{j}}{j!}$.
\end{lem}

We point out that this lemma imposes very mild assumptions on $U$ and $\beta$.
A big question in the following will be how far from optimal this estimate is
when considering the specific case \eqref{eq:1}.
\begin{proof}[Proof of Lemma \ref{lem:Duhamel}]
  We note that the evolution preserves the vanishing $x$-average of $\omega(t)$
  as does $\p_x \Delta^{-1}$ and the transport by any shear $(U(y),0)$.
  Furthermore, the transport by the shear is an $L^2$ isometry.
  Hence, the operator norm of $\beta(t)\p_x\tilde{\Delta}_{t}^{-1}$ is the same
  for all times and can be controlled by
  \begin{align*}
    \|\beta\|_{L^\infty} \|\p_x \Delta^{-1} P_{\neq}\|_{L^2 \rightarrow L^2},
  \end{align*}
  which is controlled by Poincar\'e's inequality.
  The exponential growth bound hence follows by Gronwall's lemma.

  In order to prove the convergence of the Duhamel series we need a more
  detailed description.
  Note that the coefficient functions do not depend on $x$ the problem
  decouples in frequency, so we may without loss of generality establish such a
  result for a given Fourier mode $k$ (with estimate uniform in $k$, of course).
  Let now $G_1(y,y')$ denote the Green's function of $ik(-k^2+\p_y^2)$.
  Then the kernel of $\beta(t)\p_x\tilde{\Delta}_{t}^{-1}$ frequency-localized in
  $k$ is given by
  \begin{align*}
    e^{-iktU(y)} \beta(y) G_1(y,y') e^{iktU(y')}=: e^{ikt(U(y)-U(y'))} K(y,y').
  \end{align*}
  The $j$-th term of the Duhamel iteration is thus given by
  \begin{align*}
    \int_0^t\int_{0}^{t_1}\dots \int_0^{t_{j-1}} \iint_{y_1, \dots y_j} \prod e^{ikt(U(y_{j-1})-U(y_j))} K(y_{j-1}, y_j).
  \end{align*}
  We may estimate this integral as follows:
  \begin{align*}
    \frac{t^j}{j!} \iint_{y_1, \dots y_j}\prod |K(y_{j-1}, y_j)|.
  \end{align*}
  Recalling the structure of $K(y,y')$ as $c_ke^{-|k(y-y')|}$ in the whole space
  or a function in terms of $\sinh$ and $\cosh$ in the case of an interval, we
  note that also $|K(y,y')|$ induces an operator with finite operator norm and
  that we thus obtain a bound in operator norm by
  \begin{align*}
    \frac{t^{j}}{j!} C^j,
  \end{align*}
  where $C$ is the operator norm associated to $|K(y,y')|$.
  The result hence follows by comparison with the exponential series.
\end{proof}

With this abstract convergence result at hand, we next establish a finer
description of each Duhamel iteration and the resulting value of the infinite
series.
Here we use that our equation only contains nearest-neighbor interaction and
that the $j$=th Duhamel iteration thus only relates modes that are at most $j$
apart.
We hence show that all non-trivial contributions from initial data at a mode
$(k,\eta_0)$ to a mode $(k,\eta_1)$ correspond to \emph{paths} from one mode to
the other. Furthermore, when considering the infinite Duhamel series, the
integrals are given solely in terms of the initial data.

  \begin{defi}[Path]
    \label{defi:path}
    A \emph{path} $\gamma$ from $\eta_0 \in \R$ to $\eta_1
  \in \R$ is a sequence $\gamma=(\gamma_0,
  \gamma_1, \gamma_2, \dots , \gamma_j)$ with $\gamma_0=\eta_0$,
  $\gamma_{j}=\eta_1$ and $|\gamma_{i+1}-\gamma_{i}|=1$. We call $|\gamma|:=j$
  the \emph{length} of $\gamma$.

  Given two times $t_0<t_1$ we then associate to each path
  $\gamma$ a \emph{Duhamel integral} $I_\gamma[t_0,t_1]$:
  \begin{align*}
    I_{\gamma}[t_0,t_1]= \iint_{t_0\leq \tau_0 \leq \tau_1 \leq \dots \leq \tau_j\leq t_1}  \omega(t_0,k,\gamma_0)\prod_{i=0}^{|\gamma|-1} \frac{c}{k^2+(\gamma_i-\tau_i)^2}.
  \end{align*}
  \end{defi}

\begin{thm}[Second stability theorem]
  \label{thm:secondstability}
  Suppose that $c L 2 \pi<1$.
  Then for all $\tau\geq 0$ and all $\eta$ it holds that
  \begin{align*}
    |\omega(\tau,\eta)-\omega_0(\eta)| \leq \frac{1}{1-cL 2\pi} \left(cL 2\pi \right)^{|\cdot|}*|\omega_0| (\eta), 
  \end{align*}
  where $*$ denotes the convolution in frequency $k$.
  In particular, this implies that the evolution is stable in $H^s$ for any
  $s>0$ and in Gevrey regularity.
\end{thm}

\begin{proof}
  By our assumption on $c$ it holds that
  \begin{align*}
    \int_{\R} \frac{c}{k^2+\tau^2} d\tau = \frac{c}{|k|}\pi=: d<\frac{1}{2}.
  \end{align*}
  In contrast to the norm-based approached of Theorem \ref{thm:multiplier} and Lemma
  \ref{lem:Gronwall}, we in the following consider the frequency-wise evolution
  of the solution.\\

  Suppose at first that $\omega_0=\delta_{\eta_0}$ and fix a time $\tau>0$. Then the value
  of $\omega(\eta)$ at time $\tau$ can be obtained by summing over all integrals
  corresponding to paths starting in $\eta_0$ and ending in $\eta$:
  \begin{align*}
    \omega(\tau,\eta)- \omega(0,\eta) = \sum_{\gamma: \gamma_0=\eta_0, \gamma_{|\gamma|}=\eta} I_{\gamma}[0,\tau].
  \end{align*}
  By Fubini's theorem we may easily estimate
  \begin{align*}
    |I_{\gamma}[0,\tau]| \leq d^{|\gamma|}.
  \end{align*}
  Now for any given length $j$ there are only $2^{j}$ paths starting in $\eta_0$
  of length $|\gamma|=j$ (of which only a fraction ends in $\eta$).
  If we denote the length of the shortest path connecting $\eta_0$ and $\eta$ by
  $\dist(\eta_0,\eta_1)$ it thus follows that 
  \begin{align*}
    |\omega(\tau,\eta)- \omega(0,\eta)| &\leq  \sum_{j\geq \dist(\eta_0,\eta_1)} (2d)^{j} \\
    &=\frac{1}{1-2d} (2d)^{\dist(\eta_0,\eta_1)}.
  \end{align*}

  More generally, if $\omega_0$ is not given by just a single mode but rather
  \begin{align*}
    \omega_0 = \sum_{\eta_0} \omega_0(\eta_0) \delta_{\eta_0},
  \end{align*}
  we may use the linearity of the problem to compute
  \begin{align*}
    \omega(\tau,\eta)- \omega(0,\eta) =\sum_{\eta_0} \omega_0(\eta_0) \sum_{\gamma: \gamma_0=\eta_0, \gamma_{|\gamma|}=\eta} I_{\gamma}[0,\tau].
  \end{align*}
  Therefore, by the triangle inequality it follows that 
  \begin{align*}
    |\omega(\tau,\eta)- \omega(0,\eta)| \leq \frac{1}{1-2d} (2d)^{|\cdot|}* |\omega_0(\cdot)| (\eta).
  \end{align*}
\end{proof}

\section{Instability for Long Tori and Cascades}
\label{sec:instability}

As a complementary result to the stability established in Section
\ref{sec:Duhamel}, we show that for long tori or respectively slightly larger
$c$ the dynamics are exponentially unstable.
Here, we do not construct eigenfunctions but instead establish a new cascade
mechanism for resonances with respect to the frequency in $y$.
We recall that by the Orr mechanism the multiplier associated with the stream
function
\begin{align*}
  \frac{1}{k^2+(\eta-kt)^2}
\end{align*}
is largest when $\eta-kt=0$.
In the study of the nonlinear problem \cite{deng2018},
\cite{bedrossian2015inviscid}, \cite{dengZ2019} this resonance underlies the
main growth mechanism, where $\eta$ is roughly fixed and a mode $k$ at time
$t\approx \frac{\eta}{k}$ excites a mode $k-1$, which then later excites a mode
$k-2$ and so on. As we discuss in \cite{dengZ2019} this cascade is a property of the
linearized problem around the low-frequency part of the vorticity depending on
$x$, e.g. $c \sin(x)$.
As one of the main results of this paper we show that the linearized problem
around the $x$-independent part, i.e. the shear flow component, also exhibits a
cascade mechanism but with respect to $\eta$.
That is, here $k$ is fixed and a mode $\eta$ at time $t\approx \frac{\eta}{k}$
excites the mode $\eta+1$, which later excites the mode $\eta+2$ at time
$t\approx \frac{\eta+1}{k}$ and so on.

We in particular note the following similarities and differences, where we refer
to the cascades as $k$-chain and $\eta$-chain, respectively.
\begin{itemize}
\item In the $k$-chain for any initial mode $(k_0,\eta_0)$ the resonant times are
  given by $t_{k}=\frac{\eta_0}{k}$ with $k=k_0,k_0-1,\dots, 1$.
  In particular, starting from any strictly positive time, there are only
  finitely many resonances and no resonances after a maximal time $t=\eta$.
  In contrast, the $\eta$-chain has the resonant times $t_\eta=\frac{\eta}{k_0}$
  with $\eta=\eta_0,\eta_0+1,\dots$, which is a cascade of infinite length.
\item The sequences of resonant times $t_k$ is unevenly spaced, while the
  sequence of times $t_\eta$ is a rescaled integer sequence.
\item The total growth exhibited by a $k$-chain is given by
  \begin{align*}
    \frac{c\eta_0}{k_0^2}\frac{c\eta_0}{(k_0-1)^2}\dots = \frac{(c\eta_0)^{k_0}}{(k_0!)^2},
  \end{align*}
  which attains its maximum with respect to $k_0$, $\exp(\sqrt{c\eta_0})$ for
  $k_0\approx \sqrt{c\eta_0}$. This growth factor corresponds to a Gevrey
  regularity class.
  
  The growth by the $\eta$-chain in contrast is exponential in time and tied to
  the $x$ frequency of the perturbation $k_0$ instead of the regularity.
  That is, the growth factor is given by
  \begin{align*}
    \left(2\pi \frac{c}{k}\right)^t,
  \end{align*}
  provided $2\pi \frac{c}{k}\gg 1$, even if we start with single mode initial data at
  $\eta=0$.
\item As can be seen from the growth factors and as shown in the previous
  section there is no growth in the $\eta$-chain if $\frac{c}{k}$
  is small.
  Similarly, there is only growth by a constant in the $k$-chain if
  $\sqrt{c\eta_0}$ is small.
\end{itemize}

\begin{thm}[Instability]
  \label{thm:instability}
  Let $0<c_0<\frac{1}{10}$, then there exists $L$ such that for any $c>c_0$ the
  evolution \eqref{eq:1} is exponentially unstable in $L^2$.
  More precisely, suppose that $L$ is such that $\pi c L > 20$ and $\pi c^2 L <1$
  (e.g choose $L\sim c^{-3/2}$). Then for any (arbitrarily smooth) initial data $\omega_0$ with $\omega_0(k,0)=1\geq 0.5 \max \omega_0$ it holds that
  \begin{align*}
    \omega(j-\frac{1}{2},k, j) \geq d^{j}
  \end{align*}
  for $d= \frac{\pi}{10} cL >1$ and all $j \in \N$.
\end{thm}

This shows that phase-mixing for smooth Bilipschitz profiles by itself is not strong
enough to prevent an exponential instability. Some smallness or non-resonance is
necessary. While the embedding eigenvalue criterion of \cite{Zhang2015inviscid} provides a very good
description of this, we think the present method of proof of constructing a
sustained echo chain provides an important new perspective on this instability.

\begin{proof}
  By Lemma \ref{lem:Duhamel} the Duhamel iteration converges for all times and controlled by
  an exponential series. In the following we derive more precise control of the
  Duhamel iterates as sums over all \emph{paths} (see Definition \ref{defi:path}) starting in a frequency $\eta_0$ and
  ending in a frequency $\eta_1$.

  We recall that our equation \eqref{eq:1}
  \begin{align*}
      \p_\tau \omega(\eta) +  \frac{c}{2} \frac{1}{k^2+(\eta+1-\tau)^2} \omega(\eta+1) - \frac{c}{2} \frac{1}{k^2+(\eta-1-\tau)^2} \omega(\eta-1)=0,
  \end{align*}
  considers only nearest neighbor interactions.
 
  In the following we argue by an iteration scheme, where we consider the
  sequence of times $T_j=j-\frac{1}{2}$, $j \in \N$.
  We claim that if
  \begin{align}
    \label{eq:5}
    |\tilde{\omega}(T_j,k,j)|\geq 0.5 \max_\eta |\tilde{\omega}(T_j,k,\eta)|,
  \end{align}
  then it holds that
  \begin{align}
    |\tilde{\omega}(T_{j+1},k,j+1)|\geq 0.5 \max_\eta |\tilde{\omega}(T_{j+1},k,\eta)|,
  \end{align}
  and additionally
  \begin{align}
    |\tilde{\omega}(T_{j+1},k,j+1)|\geq d |\tilde{\omega}(T_j,k,j)|.
  \end{align}
  The result then follows immediately by induction in $j$.\\

  Let thus $j \in \N$ be given and consider a given frequency $\eta_0$ and let
  again at first $\omega_0(T_j)$ be given by a single mode:
  \begin{align*}
  \tilde{\omega}(t_j,k,\eta)=\delta_{\eta_0}(\eta). 
  \end{align*}
  Then for any $\eta$ we may compute
  \begin{align*}
    \omega(T_{j+1},k,\eta)-\omega(T_{j},k,\eta)= \sum_{\gamma: \gamma_0=\eta_0, \gamma_{|\gamma|}=\eta}I_{\gamma}[T_j,T_{j+1}].
  \end{align*}
  While in the proof of Theorem \ref{thm:secondstability} we roughly estimated
  \begin{align*}
    I_{\gamma}[T_j,T_{j+1}] \leq d^{|\gamma|},
  \end{align*}
  in the following we estimate more carefully.
  In each integral we have to consider terms of the form
  \begin{align*}
    \int_{\tau_{i-1}\leq \tau_{i}\leq \tau_{i+1}} \frac{c}{k^2+(\gamma_i-\tau_{i})^2} d\tau_{i} \leq \int_{t_j\leq \tau_i\leq t_{j+1}} \frac{c}{k^2+(\gamma_i-\tau_{i})^2} d\tau_i.
  \end{align*}
  If $|\gamma_i-j|\geq 1$, we can bound this contribution by
  \begin{align*}
    \int_{-\frac{1}{2}}^{\frac{1}{2}} \frac{c}{k^2+\frac{1}{4}} d\tau_i\leq 4c=: \delta< \frac{1}{4}.
  \end{align*}
  We call such $\gamma_i$ \emph{non-resonant}.
  We recall that by Remark \ref{rem:periodic} that we only need to consider
  $\gamma_i \in \Z$ and thus there is a clean dichotomy: Either $\gamma_i\neq j$
  is non-resonant or $\gamma_i=j$ is (perfectly) resonant.

  In the following we thus group paths by the number of resonances appearing.
  Similarly to the proof of Theorem \ref{thm:secondstability}, if all $\gamma_i$
  in a path $\gamma$ are non-resonant, we may estimate
  \begin{align*}
    I_{\gamma}[T_j,T_{j+1}]\leq \delta^{|\gamma|}.
  \end{align*}
  Again roughly estimating the number of all paths of length $|\gamma|$ by
  $2^{|\gamma|}$, the sum over all such paths can be estimated by a geometric
  series:
  \begin{align}
    \label{eq:purelynonres}
    \sum_{\gamma:\text{non-resonant}}I_{\gamma}[T_j,T_{j+1}] \leq \sum_{i \geq \text{dist}(\eta,\eta_0)} 2^{i}\delta^{i} =\frac{1}{1-2\delta}(2\delta)^{\text{dist}(\eta,\eta_0)}.
  \end{align}
  We remark that here the distance refers to the shortest non-trivial path, thus
  $\dist(\eta,\eta_0)=|\eta-\eta_0|$ if $\eta\neq \eta_0$ and $\dist(\eta_0,\eta_0)=2$. \\
  
  In contrast, consider the special case $\eta_0=j$ and $\eta=j+1$, then the
  path $\gamma=(\eta_0,\eta_0+1)$ has the associated integral is given by
  \begin{align}
    \label{eq:singleres}
    \int_{-\frac{1}{2}}^{\frac{1}{2}} \frac{c}{k^2+\tau^2} d\tau = \frac{2c}{k}\arctan(\frac{1}{2k})=:r.
  \end{align}
  If $k<c$ is sufficiently small, then $2 \arctan(\frac{1}{2k})\approx \pi$
  and $r\approx c\frac{\pi}{k}\geq c \pi L >1$ is comparatively large.\\

  More generally, let $\gamma$ be a path starting in $\eta_0$ and ending in $\eta$.
  \begin{itemize}
  \item 
  If all $\gamma_i$ are non-resonant, this path is estimated by
  \eqref{eq:purelynonres}. Thus suppose that it has several resonances.
\item 
  We note that in order to be resonant, the path first has to reach $j$ starting
  from $\eta_0$, for which it needs $j_1\geq |\eta_0-j|$ steps.
\item It is then resonant for a number $j_2\geq 1$ of times, between which is
  non-resonant a number at least $j_3\geq j_2-1\geq 0$ times, since subsequent entries in
  a path are distinct, that is you have to leave the resonant frequency before
  visiting again.
\item Finally, after visiting the resonant frequency a last time,
  the path has to reach the frequency $\eta$ for which it uses $j_4\geq
  |\eta-j\pm 1|$ non-resonant steps (for example the path $\gamma=(j, j-1)$ has
  a single resonance and no non-resonance).
\end{itemize}
Thus, in total the integral corresponding to this path can be estimated by
\begin{align*}
  \delta^{j_1} r^{j_2} \delta^{j_3} \delta^{j_4}&= \delta^{j_1+j_4} r^{1} (r \delta)^{j_2-1} \delta^{j_3-j_2+1}, \\
  j_1\geq |\eta_0-j|, \ j_4 &\geq |\eta-j\pm 1|, \\
  j_2-1\geq 0, \ j_3-j_2+1&\geq 0,
\end{align*}
where we use the short notation $|\eta-j \pm 1|= \min(|\eta-j+1|, |\eta-j-1|)$.
We note that by assumption
\begin{align}
  \label{eq:6}
  \begin{split}
  r\delta \leq 8 \pi c ^2 L&<\frac{1}{4}, \\
  \delta&<\frac{1}{4}, \\
  r& >10.
  \end{split}
\end{align}
Again roughly estimating the number of such paths by $2^{|\gamma|}$, we estimate
the sums in $j_2-1$ and in $j_3-j_2+1$ by
\begin{align*}
  \frac{1}{1-2r\delta} \frac{1}{1-2\delta}\leq 4
\end{align*}
and similarly estimate the sums in $j_1,j_4$ by
\begin{align*}
  \frac{1}{(1-2\delta)^2} (2\delta)^{|\eta_0-j|+|\eta-j\pm 1|}.
\end{align*}
We thus in total obtain an estimate by
\begin{align*}
  \frac{1}{(1-2\delta)^3} \frac{1}{1-2r\delta} r (2\delta)^{|\eta_0-j|+|\eta-j\pm 1|}.
\end{align*}
In particular, we note the exponential decay in $|\eta-j\pm 1|$ and in
$|\eta_0-j|$ and that hence this bound is much smaller than $r$ unless
$\eta_0=j$ and $\eta\in \{j-1, j+1\}$.
Taking into account the single resonance \eqref{eq:singleres} we further obtain
a lower bound in that case by $r-\frac{r}{2}=\frac{r}{2}$, by estimating all error terms as $\frac{r}{2}$. \\

As in Theorem \ref{thm:secondstability} in the case of general initial data, we
rely on the linearity of the problem to reduce to the single-mode case.
Denoting $C=\frac{1}{(1-2\delta)^3} \frac{1}{1-2r\delta}$, and recalling the
estimate \eqref{eq:purelynonres} we summarize the upper bound of the map
$\omega(T_j) \mapsto \omega(T_{j+1})$ as
\begin{align*}
  |\omega(T_{j+1},\eta)-\omega(T_j,\eta)|\leq (2\delta)^{\dist(\cdot,\cdot)} *|\omega(T_j)| +  Cr (2\delta)^{|\eta-j\pm 1|} \sum_{\eta_0} (2\delta)^{|\eta_0-j|}|\omega(T_j,\eta_0)| ,
\end{align*}
where we also summed over all $\eta_0$.
In particular, it holds that
\begin{align*}
  \sup_{\eta} (2\delta)^{\dist(\cdot,\cdot)} *|\omega(T_j)| \leq \frac{1}{1-2\delta} \sup_{\eta} |\omega(T_j)|, \\
  \sum_{\eta_0} (2\delta)^{|\eta_0-j|}|\omega(T_j,\eta_0)| \leq \frac{1}{1-2\delta} \sup_{\eta_0}|\omega(T_j,\eta_0)|,
\end{align*}
which provides an upper bound on the new supremum.
On the other hand, we saw in \eqref{eq:purelynonres} that the contribution by
the path $\gamma=(j,j+1)$ (and analogously $(j,j-1)$) is given by $r$, while all
other paths ending in $j+1$ involve at least one more non-resonance.
Hence, we further obtain the following lower bound:
\begin{align*}
  & \quad |\omega(T_{j+1},j+1)-\omega(T_j,j+1)-r \omega(T_j,j)| \\
  &\leq \frac{1}{1-2\delta} \sup_{\eta} |\omega(T_j)| + Cr (2\delta)  \frac{1}{1-2\delta} \sup_{\eta_0}|\omega(T_j,\eta_0)|.
\end{align*}
Recalling that $\omega(T_j,j)$ is comparable to the supremum at that time by
assumption \eqref{eq:5} and using that $r\gg 1$ and choosing $\delta$
sufficiently small such that $Cr (2\delta)  \frac{1}{1-2\delta} \ll 1$, it
follows that
\begin{align*}
  |\omega(T_{j+1},j+1)-r \omega(T_j,j)|\leq \frac{1}{4} r \omega(T_j,j)
\end{align*}
and hence
\begin{align*}
  |\omega(T_{j+1},j+1)| &\geq \frac{3}{4} r |\omega(T_j,j)| \\
  &\geq   \frac{1}{2} \sup_{\eta} |\omega(T_{j+1},\eta)|.
\end{align*}
This concludes the proof of the claim and hence of the theorem.
\end{proof}

In this theorem we have shown that if a torus is sufficiently long such that
$cL\gg 1$, then a new resonance cascade mechanism persists, where $(k,\eta=j)$
excites a mode $(k,j+1)$ around time $T_j$, which in turn excites $(k, j+2)$ and
so on.
We thus present a new explicit instability mechanism associated to the smallness
assumption imposed in \cite{Zill5}.
Furthermore, similarly to convexity/concavity assumptions imposed in classical
stability results by Rayleigh, Fjortoft or Arnold \cite{drazin2}, this result
shows that even if a flow is bilipschitz and smooth, this is not sufficient to
establish (asymptotic) stability and that some additional condition to control
long wave-length perturbations is necessary.
We thus develop a further understanding of what (linear) instability mechanisms
may be encountered in the study of perturbations to shear flows, how
instabilities may propagate and cascade and what conditions may be imposed to
avoid these scenarios (e.g. limiting the wave length of admissible
perturbations).
Here an interesting but very challenging question concerns the implications of
this linear instability mechanism for the nonlinear problem. While such an
instability rules out some types of asymptotic convergence results (that is,
scattering to the linear dynamics), it might be the case that the linearization
around the initial velocity profile is the ``wrong guess'' for the asymptotics
and that the evolution is nevertheless asymptotically stable. 

In view of the nonlinear problem we remark that the $\eta$-chain model of the
present article and the fluid echo chains studied in \cite{dengmasmoudi2018} and
\cite{dengZ2019} share commonalities in how instabilities appear at critical
times and then lead to new instabilities at later times, resulting in a cascade.
However, as remarked earlier the details such as the direction (in Fourier space), length or time
scales of these cascades are very different.
Furthermore, in the general nonlinear problem one may expect both mechanisms to
interact, resulting in resonances in further directions.

\bibliographystyle{alpha} \bibliography{citations2}

\end{document}